\documentclass[10pt, leqno]{article}
\usepackage[utf8]{inputenc}

\usepackage{amsmath,amssymb,amsthm, epsfig}

\usepackage{amsfonts}

\usepackage{hyperref}

\usepackage{amsmath}

\usepackage{color}

\usepackage{ulem}

\usepackage{mathrsfs}

\usepackage{wasysym}

\usepackage{stackrel}

\title{\Large{\bf  
Sharp regularity for degenerate fully nonlinear equations with oblique boundary conditions and Hamiltonian terms}}
\author{\it by \smallskip \\ 
Junior da Silva Bessa
\quad \&\quad
Gleydson C. Ricarte
}
\date{}
\newlength{\hchng}
\newlength{\vchng}
\def \supp {\mathrm{supp } }

\newcommand{\defeq}{\mathrel{\mathop:}=}

\newtheorem{theorem}{Theorem}[section]

\newtheorem{lemma}[theorem]{Lemma}

\newtheorem{corollary}[theorem]{Corollary}

\theoremstyle{definition}

\newtheorem{definition}[theorem]{Definition}

\newtheorem{example}[theorem]{Example}

\theoremstyle{remark}

\newtheorem{remark}[theorem]{Remark}

\numberwithin{equation}{section}

\newcommand{\intav}[1]{\mathchoice {\mathop{\vrule width 6pt height 3 pt depth  -2.5pt
\kern -8pt \intop}\nolimits_{\kern -6pt#1}} {\mathop{\vrule width
5pt height 3  pt depth -2.6pt \kern -6pt \intop}\nolimits_{#1}}
{\mathop{\vrule width 5pt height 3 pt depth -2.6pt \kern -6pt
\intop}\nolimits_{#1}} {\mathop{\vrule width 5pt height 3 pt depth
-2.6pt \kern -6pt \intop}\nolimits_{#1}}}

\begin{document}
\maketitle

\begin{abstract}
\noindent We prove optimal boundary  $C^{1,\alpha}$ regularity for viscosity solutions of degenerate fully nonlinear uniformly elliptic equations with oblique boundary conditions and Hamiltonian terms of the form
\[
\begin{cases}
|Du|^{\gamma}F(D^2 u) + \varrho(x)|Du|^{\sigma} = f(x) & \text{in } \Omega,\\
\beta(x)\cdot Du+\zeta(x)u = g(x) & \text{on } \partial \Omega,
\end{cases}
\]
where $\gamma>0$ and $0<\sigma\le 1+\gamma$. We develop a compactness framework for affine translations, linking the size of the translation to the Hamiltonian structure. This is combined with a boundary improvement-of-flatness argument adapted to oblique boundary data, yielding the optimal boundary regularity.


\vspace{0.2cm}

\end{abstract}
%\tableofcontents

\section{Introduction}

This paper is concerned with regularity results up to the boundary for viscosity solutions to a degenerate fully nonlinear elliptic equation with an oblique boundary condition
\begin{equation}\label{1.1}
\left\{
\begin{array}{rclcl}
\mathcal{G}(D^{2}u,Du,x)&=& f(x)& \mbox{in} &   \Omega,  \\
\mathcal{B}(Du,u,x)&=& g(x) &\mbox{on}& \partial \Omega,
\end{array}
\right.	
\end{equation}
where $\Omega \subset \mathbb{R}^n$ is a bounded $C^1$ domain. Here, the operator $\mathcal{G}:\mathrm{Sym}(n)\times \mathbb{R}^{n}\times \Omega\to \mathbb{R}$ (where $\mathrm{Sym}(n)$ denotes the space of $n\times n$ real symmetric matrices) is defined by
\begin{equation}\label{1.2}
\mathcal{G}(\mathrm{M},\vec{\xi},x)\defeq |\vec\xi|^{\gamma}F(\mathrm{M}) + \varrho(x)|\vec{\xi}|^{\sigma},
\end{equation}
where $F=F(\mathrm{M})$ is uniformly elliptic (cf. the structural condition $\mathbf{(H1)}$), $\varrho$ is a bounded function, $\gamma>0$, and $0<\sigma\leq 1+\gamma$.  This model endows the equation with a term that degenerates in the region where $\{|Du|=0\}$, while at the same time incorporating a Hamiltonian term of the form
\[
\mathcal{H}(x,\vec{\xi})=\varrho(x)|\vec{\xi}|^{\sigma},
\]
which exhibits sublinear-to-linear homogeneity with respect to the gradient variable. This feature requires a delicate analysis of viscosity solutions and their regularity for such models. The study of problems associated with the operator $\mathcal{G}$ arises naturally in the optimal stochastic control problems \cite{BirDemLeo19Erg,LarLio89} and as a generalization of Hamilton-Jacobi equations \cite{ArmTra15,JesPimUrb23}.

The operator $\mathcal{B}:\mathbb{R}^{n}\times\mathbb{R}\times \partial \Omega \longrightarrow \mathbb{R}$ is defined by
\begin{equation*}
\mathcal{B}(\vec{\xi},s,x)\defeq \beta(x)\cdot \vec{\xi}+\zeta(x)s,
\end{equation*}
where $\zeta$ is a continuous real-valued function. The vector field $\beta$ is assumed to be continuous and to satisfy the obliqueness condition. More precisely, we assume that there exists a constant $\mu_{0}>0$ such that
\begin{equation}\label{1.3}
\beta\cdot \vec{\mathbf{n}}\geq \mu_{0}\quad\text{and}\quad \|\beta\|_{L^{\infty}(\partial\Omega)}\leq 1,
\end{equation}
where $\vec{\mathbf{n}}$ denotes the inner unit normal vector to $\partial \Omega$. In this case, we refer to $\mathcal{B}$ as an \textit{oblique boundary operator}.

Under suitable regularity assumptions on the boundary data $\beta$, $\zeta$, and $g$, as well as on the source term $f$ (see the structural hypotheses {\bf(H1)-(H2)}), we establish optimal $C^{1,\alpha}$ regularity up to the boundary for solutions to problem \eqref{1.1}.

The interest in studying models with oblique boundary conditions is naturally motivated by the modeling of several phenomena, such as reflected shocks in transonic flow \cite{CKL00}, the theory of capillary problems \cite{ESY96,FinnLuli}, Brownian motion \cite{EFH,RW}, stochastic optimal control of reflected diffusion processes \cite{LioTru86}, and gravity field determination from boundary observations \cite{DDO06}. Regarding H\"older regularity results for this type of problem, we refer to \cite{BesdaSRic262,BesdaSRic26,ChoKim23,DonLi20,DonPan21,LiZhang} for a modern and non-exhaustive list of related works.

The regularity theory for degenerate fully nonlinear equations with Hamiltonian terms has attracted considerable attention. We begin by recalling the seminal works of Birindelli and Demengel, who established $C^{1,\alpha}$ regularity results for models with linear Hamiltonian terms \cite{BirindelliDemengel14} (see also  \cite{BezdaSil25} and \cite{Ndaw23} ) and for degenerate/singular models with sublinear-to-linear Hamiltonian terms \cite{BirindelliDemengel16} (see also \cite{BirDemLeo}). These authors, together with Leoni, extended these results to operators with superlinear terms of at most subquadratic growth \cite{BirDemLeo19}. More recently, Andrade and Nascimento \cite{PT} studied interior estimates for viscosity solutions to
\begin{equation*}
\Phi(x,|Du|)F(x,D^{2}u)+h(x)|Du|^{m}=f(x),
\end{equation*}
under the assumption that $\Phi(x,|t|)\approx |t|^{\gamma}$ for some $\gamma>0$ and $0<m\leq 1+\gamma$. The authors established local H\"older regularity of solutions via the Ishii--Lions method and geometric tangential approach, and obtained optimal $C^{1,\alpha}$ regularity for solutions to this problem.

In contrast with the literature mentioned above, when one considers problems with oblique boundary conditions, advances have been obtained only in the absence of Hamiltonian terms. When $\varrho\equiv 0$ in \eqref{1.2}, boundary $C^{1,\alpha}$ regularity results were established by Banerjee and Verma \cite{BanerjeeVerma} in the setting of Robin boundary conditions (that is, when $\beta=\vec{\mathbf{n}}$), and optimal regularity in this framework was later obtained by Ricarte \cite{Ricarte}. Still in the absence of Hamiltonian terms, Byun, Kim, and Oh \cite{BKO} proved $C^{1,\alpha}$ regularity for degenerate models with oblique boundary conditions. Concerning this problem, Bessa, Ricarte, and Silva \cite{BesRicSil26} extended this result for the optimal gradient regularity theory to the degenerate oblique setting under relaxed convexity assumptions. Bessa and Oh \cite{BesOh26} extended this latter result to the setting of two-phase problems of the form
$$
\left\{
\begin{array}{rclcl}
 (|Du|^{p}+\mathfrak{a}(x)|Du|^{q})F(D^2 u)&=& f(x)& \mbox{in} &   \Omega  \\
 \beta(x)\cdot Du+\zeta(x) u&=& g(x) &\mbox{on}& \partial \Omega,
\end{array}
\right.
$$
where $0<p\leq q$ and $\mathfrak{a}$ is a nonnegative function. We emphasize that, in their approach, a compactness result was established for the model translated by affine profiles by analyzing the size of the translation, which in turn yields optimal estimates for this model. The authors also derived several applications, including the analysis of the case when the operator $F$ is nearly convex/nearly concave and estimates in the homogeneous setting, that is, when the source term vanishes. More recently, Byun, Kim, and Kim further extended the estimates in \cite{BKK26} to the following problem
$$
\left\{
\begin{array}{rclcl}
\Phi(x,|Du|)F(D^{2}u)&=& f(x)& \mbox{in} &   \Omega  \\
 \beta(x)\cdot Du+\zeta(x) u&=& g(x) &\mbox{on}& \partial \Omega,
\end{array}
\right.
$$
with $\Phi$ exhibiting a singular/degenerate behavior. In this case, they also carried out an analysis of the translated problem by controlling the size of the translation and the balance of the function $\Phi$ with respect to its second variable. 

Finally, despite the substantial progress achieved in the literature, there remains a gap in the regularity theory for degenerate models with Hamiltonian terms under oblique boundary conditions. Indeed, interior and boundary regularity results are available for degenerate equations with Hamiltonian terms \cite{PT,BirindelliDemengel14,BezdaSil25}, while degenerate problems with oblique boundary conditions have been addressed in \cite{BesOh26,BKK26,BKO}. However, to the best of our knowledge, no results are available when both features are present simultaneously. This paper aims to bridge this gap.

In pursuit of our goal, the analytical tools are inspired by \cite{PT} and \cite{BKO}, particularly regarding the compactness analysis of solutions via translations and the control of such translations together with the Hamiltonian term. However, our approach requires a finer analysis of the operator $\mathcal{G}$ when performing affine translations, in combination with the oblique boundary data. With this in mind, our strategy follows the outline below:
\begin{itemize}
\item[I.] First, we establish H\"{o}lder regularity for solutions to problem \eqref{1.1} translated by a vector $\vec{q}\in\mathbb{R}^{n}$ and in the case where the operator $\mathcal{B}$ does not depend on the second variable (Theorem \ref{Holderregtransl}).
\item[II.] Under suitable smallness assumptions on the data, using geometric tangential methods, we ensure that the solutions to the problem in Step~I are close to the solutions of a uniformly elliptic homogeneous equation with constant oblique boundary condition (Theorem \ref{stab}).
\item[III.] Finally, we employ an improvement of flatness result (Lemma \ref{key 1}) and provide a proof of our first main theorem (Theorem \ref{main}), which yields the optimal regularity in the case where the absorption term $\zeta(x)u$ vanishes. Building upon this result, and using a suitable translation of the source term in the boundary datum, we obtain the desired regularity for the general model \eqref{1.1} (Corollary \ref{main1}).  
\end{itemize}

The paper is divided into three further sections. In Section \ref{Section2}, we introduce the basic notation, present the main results, and collect some preliminary material. Section \ref{Section3} is concerned with the compactness of solutions to problem \eqref{1.1} under affine translations. Finally, Section \ref{Section4} is devoted to the development of our approach and to the proofs of Theorem \ref{main} and Corollary \ref{main1}.

\section{Preliminaries and Main results}\label{Section2}

We begin by introducing the notation that will be used throughout the paper:
\begin{itemize}
\item For $r>0$ we write $\mathrm{B}_r=\{|x|< r\}$ the ball with center $0$ and radius $r$.
\item Given $r>0$, $\mathrm{B}^+_r = \mathrm{B}_r \cap \{x_n>0\}$ the half-ball and $\mathrm{T}_r = \mathrm{B}_r \cap \{x_n=0\}$ the flat boundary of $\mathrm{B}^+_r$.
\item For $x_0 \in \mathbb{R}^n$ we write $\mathrm{B}_r(x_0)=x_0 + \mathrm{B}_r$ the ball with center $x_0$ and radius $r>0$.
\item For a domain $\Omega \subset \mathbb{R}^n$ and $x_0 \in \mathbb{R}^n$ we write $\Omega_r(x_0)=\Omega \cap \mathrm{B}_r(x_0)$ and $\partial \Omega_r(x_0) =\partial \Omega \cap \mathrm{B}_r(x_0)$. In particular, when $x_{0}=0$, we simply write $\Omega_r $ and $\partial \Omega_r$.
\end{itemize}

Throughout this paper, we will assume the following structural conditions on the
problem \eqref{1.1}:
\begin{enumerate}
\item [\bf(H1)]\label{HypA1} ({\bf Ellipticity}) We assume that $F:\text{Sym}(n)\longrightarrow\mathbb{R}$ is a uniformly elliptic operator, i.e., there exist positive constants $\lambda \le \Lambda$ such that
\begin{equation*}\label{unfellip}
\mathcal{M}^{-}_{\lambda,\Lambda}(\mathrm{X}-\mathrm{Y})  \le F(\mathrm{X})-F(\mathrm{Y}) \le \mathcal{M}^{+}_{\lambda, \Lambda}(\mathrm{X}-\mathrm{Y})
\end{equation*}
for all $\mathrm{X}, \mathrm{Y} \in \textrm{Sym}(n)$. Here, $\mathcal{M}^{\pm}_{\lambda,\Lambda}$ are the Pucci extremal operators, defined by
\begin{equation*}
\mathcal{M}^{+}_{\lambda,\Lambda}(\mathrm{X})=\Lambda \sum_{e_{i}>0}e_{i}+\lambda \sum_{e_{i}<0}e_{i}\quad \text{and}\quad \mathcal{M}^{-}_{\lambda,\Lambda}(\mathrm{X})=\lambda \sum_{e_{i}>0}e_{i}+\Lambda \sum_{e_{i}<0}e_{i}
\end{equation*}
where $e_{i}=e_{i}(\mathrm{X})$ ($1\leq i\leq n$) are the eigenvalues of the matrix $\mathrm{X}$. Moreover, for normalization purposes, we shall always assume that $F(0)=0$, which is not restrictive. 
	
\item[\bf (H2)] ({\bf Regularity of the data}) \, The data satisfy $\varrho,f \in C^0(\Omega)\cap L^{\infty}(\Omega)$, $\zeta, g \in C^{0,\alpha}(\partial\Omega)$, and $\beta \in C^{0,\alpha}(\partial \Omega; \mathbb{R}^n)$ for some $\alpha\in (0,\alpha_{0}]$, where $\alpha_{0}\in (0,1]$ is the exponent from the $C^{1,\alpha_0}$ regularity theory for the homogeneous equation with constant oblique boundary condition of the form
\begin{equation}\label{homprob}
\left\{
\begin{array}{rclcl}
 F(D^2 v)  &=& 0& \mbox{in} &   \mathrm{B}^{+}_{1}  \\
 \beta_0\cdot D v &=& 0 &\mbox{on}& \mathrm{T}_1,
\end{array}
\right.	
\end{equation}
where $\beta_0$ is a constant oblique vector satisfying \eqref{1.3}, with the estimate
$$
\|v\|_{C^{1,\alpha_0}(\overline{\mathrm{B}^{+}_{1/2}})} \le \mathrm{C}_{\star} \|v\|_{L^{\infty}(\mathrm{B}^+_1)}
$$
for some constant $\mathrm{C}_{\star} =\mathrm{C}_{\star}(n,\lambda,\Lambda, \mu_0)>0$.
\end{enumerate}
\begin{remark}\label{Remark2.1}
Regarding the exponent $\alpha_{0}$ in the structural condition {\bf (H2)}, D. Li and K. Zhang \cite[Theorem 1.2]{LiZhang} proved that it indeed exists under these assumptions, and that $\alpha_{0}=1$ when $F$ is concave or convex \cite[Theorem 1.3]{LiZhang}. Bessa, Ricarte, and Silva \cite[Proposition 5.1]{BesRicSil26} showed that $\alpha_{0}=1$ also holds for a more general class of operators, namely, quasiconvex/quasiconcave (see \eqref{defquasi}) ones.
\end{remark}

Our first result establishes the $C^{1,\alpha'}$ regularity, where $\alpha'$ is the optimal regularity exponent determined by the regularity exponent $\alpha_{0}$ associated with the homogeneous problem \eqref{homprob}, the H\"{o}lder regularity exponent $\alpha$ of the boundary data, and the intrinsic scaling exponent. Specifically, the regularity exponent $\alpha'$ is characterized as follows:
\[
\alpha'=\min\left\{\alpha_{0}^{-},\alpha,\frac{1}{1+\gamma}\right\},
\]
where $\alpha_{0}^{-}$ indicates that the endpoint $\alpha_{0}$ may not be attained. More precisely, we have
\begin{eqnarray*}
\left\{
\begin{array}{rlcl}
\mbox{If} & \min\left\{\alpha,\frac{1}{1+\gamma}\right\}<\alpha_{0}
& \mbox{then} & u\in C^{1,\min\left\{\alpha,\frac{1}{1+\gamma}\right\}},\\
\mbox{If} & \min\left\{\alpha,\frac{1}{1+\gamma}\right\}\geq\alpha_{0}
& \mbox{then} & u\in C^{1,\nu}\ \mbox{for any}\ 0<\nu<\alpha_{0}.
\end{array}
\right.
\end{eqnarray*}

With this in mind, we can state our first main result.
\begin{theorem}\label{main}
Let $\Omega \subset \mathbb{R}^{n}$ be a bounded $C^{1}$ domain, $0 \in \partial \Omega_1$ and $\sigma,\gamma$ constants such that $0<\sigma \le 1+\gamma$. Suppose that $u$ is a viscosity solution of
$$
\left\{
\begin{array}{rclcl}
\mathcal{G}(D^{2}u,Du,x)&=& f(x)& \mbox{in} &   \Omega_1  \\
\mathcal{B}(Du,0,x)&=& g(x) &\mbox{on}& \partial \Omega_1.
\end{array}
\right.	
$$
Assume that the structural conditions {\bf (H1)-(H2)} are valid. Then $u \in C^{1,\alpha^{\prime}}(\overline{\Omega_{1/2}})$ for 
\begin{equation*}
\alpha^{\prime}=\min\left\{\alpha_{0}^{-},\alpha,\frac{1}{1+\gamma}\right\}.
\end{equation*}
Moreover, the following estimates hold: 
\begin{enumerate}
\item[a)]For $\sigma < 1+\gamma$,
$$
\|u\|_{C^{1,\alpha^{\prime}}(\overline{\Omega_{1/2}})} \le \mathrm{C} \left( \|u\|_{L^{\infty}(\overline{\Omega_1})} + \|\varrho\|^{\frac{1}{1+\gamma-\sigma}}_{L^{\infty}(\overline{\Omega_1})} +\|f\|^{\frac{1}{1+\gamma}}_{L^{\infty}(\overline{\Omega_1})} + \|g\|_{C^{0,\alpha}(\partial \Omega_1)} \right),
$$
where $\mathrm{C}>0$ depends only on $n,\lambda,\Lambda,\alpha, \gamma,\sigma, \mu_0, [\beta]_{C^{0,\alpha}(\partial \Omega_1)}$, and $C^1$ modulus of $\partial \Omega_1$.
\item[b)] For $\sigma=1+\gamma$,
$$
\|u\|_{C^{1,\alpha^{\prime}}(\overline{\Omega_{1/2}})} \le \overline{\mathrm{C}} \cdot \left( \|u\|_{L^{\infty}(\overline{\Omega_1})}+ \|f\|_{L^{\infty}(\overline{\Omega_{1}})}^{\frac{1}{1+\gamma}} + \|g\|_{C^{0,\alpha}(\partial \Omega_1)}\right),
$$
where $\overline{\mathrm{C}}>0$ depends additionally on  $\|\varrho\|_{L^{\infty}(\overline{\Omega_1})}$.
\end{enumerate}
\end{theorem}

From Theorem \ref{main}, we obtain boundary regularity for the general model of the form \eqref{1.1}.

\begin{corollary}\label{main1}
Let $\Omega \subset \mathbb{R}^{n}$ be a bounded $C^{1}$ domain, $0 \in \partial \Omega_1$ and $\sigma,\gamma$ constants such that $0<\sigma \le 1+\gamma$. Suppose that $u$ is a viscosity solution of
\begin{equation}\label{2.1}
\left\{
\begin{array}{rclcl}
\mathcal{G}(D^{2}u,Du,x) &=& f(x)& \mbox{in} &   \Omega_1  \\
 \mathcal{B}(Du,u,x)&=& g(x) &\mbox{on}& \partial \Omega_1.
\end{array}
\right.	
\end{equation}
Assume that the structural conditions {\bf (H1)-(H2)} are valid. Then $u \in C^{1,\alpha^{\prime}}(\overline{\Omega_{1/5}})$ for 
\begin{equation*}
\alpha^{\prime}=\min\left\{\alpha_{0}^{-},\alpha,\frac{1}{1+\gamma}\right\}.
\end{equation*}
Moreover, the following estimates hold: 
\begin{enumerate}
\item[a)]For $\sigma < 1+\gamma$,
$$
\|u\|_{C^{1,\alpha^{\prime}}(\overline{\Omega_{1/5}})} \le \mathrm{C} \cdot \left( \|u\|_{L^{\infty}(\overline{\Omega_1})} + \|\varrho\|^{\frac{1}{1+\gamma-\sigma}}_{L^{\infty}(\overline{\Omega_1})} +\|f\|^{\frac{1}{1+\gamma}}_{L^{\infty}(\overline{\Omega_1})} + \|g\|_{C^{0,\alpha}(\partial \Omega_1)} \right)
$$
where $\mathrm{C}>0$ depends only on $n,\lambda,\Lambda,\alpha, \gamma,\sigma, \mu_0, [\beta]_{C^{0,\alpha}(\partial \Omega_1)}, \|\zeta\|_{C^{0,\alpha}(\partial \Omega_{1})}$ and $C^1$ modulus of $\partial \Omega_1$.
\item[b)] For $\sigma=1+\gamma$,
$$
\|u\|_{C^{1,\alpha^{\prime}}(\overline{\Omega_{1/5}})} \le \overline{\mathrm{C}} \cdot \left( \|u\|_{L^{\infty}(\overline{\Omega_1})}  + \|f\|_{L^{\infty}(\overline{\Omega_{1}})}^{\frac{1}{1+\gamma}}+\|g\|_{C^{0,\alpha}(\partial \Omega_1)}\right)
$$
where $\overline{\mathrm{C}}>0$ depends additionally on $\|\varrho\|_{L^{\infty}(\Omega_1)}$.
\end{enumerate}	
\end{corollary}

An interesting feature of the results above is that the $C^{1,\alpha}$ regularity exponent arises from the competition between three a priori independent quantities: the regularity exponent for the homogeneous problem with constant oblique boundary conditions, the Hölder regularity exponent of the boundary data, and the value $\frac{1}{1+\gamma}$, which stems from the intrinsic scaling of the equation associated with the sublinear-to-linear homogeneity regime of the operator $\mathcal{G}$.

The structural conditions above ensure that our results generalize the $C^{1,\alpha}$ regularity results for degenerate models with oblique boundary conditions \cite{BesRicSil26,BKO, Ricarte} and can be regarded as an extension of the estimates established only for fully nonlinear degenerate elliptic equations with Hamiltonian terms in \cite{PT,BirindelliDemengel14,BirindelliDemengel16,BirDemLeo,Ndaw23} and optimal estimates as in \cite{ART15}.

\begin{example}[\bf Sharpness]
Now we illustrate the optimality of the exponent $\alpha^{\prime}$ in the above results. 
Let $\phi \in C_{0}^{\infty}(\mathbb{R}_{+})$ be such that $\phi|_{[0,1]} \equiv 0$ and $\supp(\phi)\subset [0,2]$. 
Consider the bounded $C^{1}$ domain
\[
\Omega=\{x\in \mathrm{B}_{2};\ x_{n}>-\phi(|x|)\},
\]
and define $u(x)=|x|^{1+\alpha}$. Then $u$ is a viscosity solution to
\begin{equation*}
\left\{
\begin{array}{rclcl}
|Du|^{\gamma}\Delta u+\varrho(x)|Du|^{\sigma} &=& f(x) & \mbox{in} & \Omega_1=\mathrm{B}^{+}_{1},  \\[4pt]
u_{\vec{\mathbf{n}}}+\zeta(x)u &=& g(x) & \mbox{on} & \partial \Omega_1=\mathrm{T}_{1},
\end{array}
\right.	
\end{equation*}
where $0
<\sigma\leq 1+\gamma$, $\varrho\equiv 1$, and $\zeta\equiv 1$. Moreover,
\[
f(x)\lesssim \mathrm{C}(n,\alpha,\gamma,\sigma)
\Big(|x|^{\alpha(1+\gamma)-1}+|x|^{\alpha\sigma}\Big),
\]
and
\[
g(x)=\frac{(1+\alpha)\phi'(|x'|)}{\sqrt{1+(\phi'(|x'|))^{2}}}|x^{\prime}|^{\alpha}
+|x'|^{1+\alpha}.
\]
In particular, if $\alpha=\frac{1}{1+\gamma}$, then $f\in L^{\infty}(\Omega_1)$ and $u$ is not more regular than $C^{1,\alpha}$.
\end{example}

An immediate consequence of Corollary \ref{main1} is the regularity analysis under the quasiconvexity/quasiconcavity regime of the governing operator. Recall that an operator $F:\mathrm{Sym}(n)\to\mathbb{R}$ is said to be \textit{quasiconvex} (respectively, \textit{quasiconcave}) if, for every $c\in\mathbb{R}$, the sublevel (resp., superlevel) set
\begin{equation}\label{defquasi}
\{\mathrm{M}\in \mathrm{Sym}(n): F(\mathrm{M})\le c\}\,\, (\text{resp., } \{\mathrm{M}\in \mathrm{Sym}(n): F(\mathrm{M})\ge c\})
\end{equation}
is convex. 

We note that convex (respectively, concave) operators are quasiconvex (respectively, quasiconcave). However, this class is strictly larger. For instance, the operator
\[
F(\mathrm{M}):=\sum_{i=1}^{n}\arctan\big(e_i(\mathrm{M})\big), \qquad\mathrm{M}\in \mathrm{Sym}(n),
\]
which arises in the special Lagrangian equation, is quasiconcave (in the supercritical phase), but it is not concave in general (cf. \cite{Gof24}).

In this class, we have the following consequence.
\begin{corollary}
Let $u$ be a viscosity solution to \eqref{2.1}. Assume the hypotheses of Corollary \ref{main1} with $\alpha\geq \frac{1}{1+\gamma}$ and, in addition, that $F$ is a quasiconvex/quasiconcave operator. Then $u\in C^{1,\frac{1}{1+\gamma}}(\overline{\Omega_{1/2}})$ and the following estimates hold:
\begin{enumerate}
\item[a)]For $\sigma < 1+\gamma$,
$$
\|u\|_{C^{1,\frac{1}{1+\gamma}}(\overline{\Omega_{1/2}})} \le \mathrm{C} \left( \|u\|_{L^{\infty}(\overline{\Omega_1})} + \|\varrho\|^{\frac{1}{1+\gamma-\sigma}}_{L^{\infty}(\overline{\Omega_1})} +\|f\|^{\frac{1}{1+\gamma}}_{L^{\infty}(\overline{\Omega_1})} + \|g\|_{C^{0,\alpha}(\partial \Omega_1)} \right)
$$
where $\mathrm{C}>0$ depends only on $n,\lambda,\Lambda,\alpha, \gamma, \sigma,\mu_0, [\beta]_{C^{0,\alpha}(\partial \Omega_1)}$ and $C^1$ modulus of $\partial \Omega_1$.
\item[b)] For $\sigma=1+\gamma$,
$$
\|u\|_{C^{1,\frac{1}{1+\gamma}}(\overline{\Omega_{1/2}})} \le \overline{\mathrm{C}} \cdot \left( \|u\|_{L^{\infty}(\overline{\Omega_1})}  +\|f\|^{\frac{1}{1+\gamma}}_{L^{\infty}(\overline{\Omega_1})}+ \|g\|_{C^{0,\alpha}(\partial \Omega_1)}\right)
$$
where $\overline{\mathrm{C}}>0$ depends additionally on $\|\varrho\|_{L^{\infty}(\overline{\Omega_1})}$.
\end{enumerate}
\end{corollary}

\begin{proof}
The proof follows directly from the following observations. By Remark \ref{Remark2.1}, we know that $\alpha_{0}=1$, and since $\gamma>0$ it follows that
\[
\alpha^{\prime}=\min\left\{\alpha_{0}^{-},\alpha,\frac{1}{1+\gamma}\right\}=\frac{1}{1+\gamma}.
\]
Therefore, the claimed regularity follows from Corollary \ref{main1}.
\end{proof}

Finalizing this section, we introduce the notion of viscosity solution for fully nonlinear equations with a double degeneracy law and oblique boundary condition:
\begin{equation}\label{model}
\left\{
\begin{array}{rcll}
\mathcal{G}(D^2u,Du,x) &=& f(x) & \mbox{in } \Omega, \\
\mathcal{B}(Du,u,x) &=& g(x) & \mbox{on } \Gamma\subset \partial \Omega,
\end{array}
\right.
\end{equation}
where $\Gamma$ is a relatively open subset of $\partial \Omega$.

\begin{definition}[\bf Viscosity solution-\cite{Caff1,LiZhang}]
Let $\mathcal{G}$ be continuous in all variables, and assume that $f\in C^{0}(\Omega)$. We say that a function $u \in C^{0}(\Omega\cup \Gamma)$ is a viscosity subsolution (resp. supersolution) to problem \eqref{model} if, for any $x_{0} \in \Omega\cup\Gamma$ and any test function $\varphi \in C^{2}(\Omega\cup\Gamma)$ such that $u - \varphi$ attains a local maximum (resp. minimum) at $x_{0}$, one has
\begin{equation*}
\left\{
\begin{array}{rcll}
 \mathcal{G}(D^{2}\varphi(x_{0}),D\varphi(x_{0}),x_{0}) &\leq\ (\geq)&  f(x_{0})& \mbox{if } x_{0}\in\Omega, \\
 \mathcal{B}(D\varphi(x_{0}),\varphi(x_{0}),x_{0})&\leq\ (\geq)& g(x_{0}) &\mbox{if } x_{0}\in\Gamma.
\end{array}
\right.
\end{equation*}
We say that $u$ is a viscosity solution of problem \eqref{model} if it is both a viscosity subsolution and a viscosity supersolution.
\end{definition}

\section{Translated problem: Compactness}\label{Section3}

Here we introduce a key tool for our approach, namely, the compactness of viscosity solutions to the translated problem under an affine perturbation.

To achieve this goal, throughout the statement of the forthcoming results and the remainder of the manuscript, we shall work in an ``almost $C^{1}$-flat'' domain $\Omega_{1}$. More precisely, as in \cite{BKO}, we assume that $0\in \partial\Omega_{1}$ and that the boundary $\partial\Omega_{1}$ can be locally described as the graph of a $C^{1}$ function (recall that, in our main results, $\Omega$ is a $C^{1}$ domain) $\psi=\psi_{\Omega}:\mathrm{T}_{1}\to\mathbb{R}$. Namely,
\begin{equation*}
\partial\Omega_{1}=\{(x',x_{n})\in \mathrm{B}_{1}\,:\,x_{n}=\psi(x')\}
\qquad \text{and}\qquad 
\{(x',x_{n})\in \mathrm{B}_{1}\,:\,x_{n}>\psi(x')\}\subset\Omega_{1}
\end{equation*}
Under this change of coordinates, we employ the usual H\"older space notation for the boundary data $\beta$, $\zeta$, and $g$. For instance, we write $g\in C^{0,\alpha}(\mathrm{T}_{1})$ and identify $g(x)$ with $g(x',\psi(x'))$.  

Observe that the inward unit normal vector to $\partial\Omega_{1}$ at the point $(x',\psi(x'))$ is given by
\begin{equation*}
\vec{\mathbf{n}}(x',\psi(x'))=
\frac{1}{\sqrt{1+|D\psi(x')|^{2}}}\,(-D\psi(x'),1).
\end{equation*}
Moreover, if $[\psi]_{C^{1}(\mathrm{T}_{1})}\leq \frac{\mu_{0}}{2}$, then the obliqueness condition $\beta\cdot \vec{\mathbf{n}}\geq \mu_{0}$ implies that $\beta_{n}\geq \frac{\mu_{0}}{2}$, where $\beta_{1},\ldots,\beta_{n}$ are the coordinates functions of the vector field $\beta$.

Therefore, instead of imposing the original oblique boundary condition, we shall always assume that the vector field $\beta$ and the function $\psi$ satisfy
\begin{equation*}
\psi(0)=0,
\qquad 
[\psi]_{C^{1}(\mathrm{T}_{1})}\leq \frac{\mu_{0}}{2},
\qquad 
\beta_{n}\geq \frac{\mu_{0}}{2},
\qquad 
\text{and}
\qquad 
\|\beta\|_{L^{\infty}(\mathrm{T}_{1})}\leq 1.
\end{equation*}

After this preliminary remark, we return to our main focus. More precisely, let us consider the following problem translated by a vector $\vec{q}\in\mathbb{R}^{n}$:
\begin{equation} \label{Translated problem}
\left\{
\begin{array}{rclcl}
\mathcal{G}(D^{2}u,D u-\vec{q},x) &=& f(x) & \mbox{in} & \Omega_r, \\
\mathcal{B}(D u,0,x)&=& g(x) & \mbox{on} & \partial\Omega_{r}.
\end{array}
\right.
\end{equation}
Our goal is to establish H\"older regularity for solutions to this problem, with estimates that are independent of the vector $\vec{q}$. For this, we introduce the following notation for \(0<\lambda\leq\Lambda\):
\begin{eqnarray*}
\mathcal{P}^{+}_{\lambda,\Lambda}(D^{2}u,Du)&=&\Lambda\operatorname{tr}(D^{2}u)^{+}-\lambda\operatorname{tr}(D^{2}u)^{-}+\Lambda|Du|,\\
\mathcal{P}^{-}_{\lambda,\Lambda}(D^{2}u,Du)&=&\lambda\operatorname{tr}(D^{2}u)^{+}-\Lambda\operatorname{tr}(D^{2}u)^{-}-\Lambda|Du|.
\end{eqnarray*}

Concerning compactness, we have the following result proved in \cite[Theorem 3.1]{BKO}.
\begin{theorem}\label{Holderreg}
Assume that \(\Omega\) is a \(C^{1}\)-domain. Then there exists a small \(\mu=\mu(\mu_{0})\leq \frac{\mu_{0}}{2}\) such that if \([\psi_{\Omega}]_{C^{1}}\leq \mu\), for any \(u\in C^{0}(\overline{\Omega_{1}})\) satisfying, in the viscosity sense,
\begin{equation}\label{3.2}
\left\{
\begin{array}{rclcl}
\mathcal{P}^{-}_{\lambda,\Lambda}(D^{2}u,Du) &\leq& \mathrm{C}_{0} & \mbox{in} & \{|Du-\vec{q}|>\theta\}\cap \Omega_1, \\
\mathcal{P}^{+}_{\lambda,\Lambda}(D^{2}u,Du) &\geq& -\mathrm{C}_{0} & \mbox{in} & \{|Du-\vec{q}|>\theta\}\cap \Omega_1, \\
\mathcal{B}(Du,0,x)&=& g(x) & \mbox{on} & \partial\Omega_{1},\\
\|u\|_{L^{\infty}(\Omega_{1})} &\leq& 1, &  & \\
\|g\|_{L^{\infty}(\mathrm{T}_{1})} &\leq& \mathrm{C}_{0}, &  & 
\end{array}
\right.
\end{equation}
for some \(0<\theta\leq 1\), then \(u\in C^{0,\bar{\alpha}}(\overline{\Omega_{\frac{1}{2}}})\) for some \(\bar{\alpha}=\bar{\alpha}(n,\lambda,\Lambda,\mu_{0})\in(0,1)\). Moreover,
\begin{equation*}
\|u\|_{C^{0,\bar{\alpha}}(\overline{\Omega_{\frac{1}{2}}})}\leq \mathrm{C}(n,\lambda,\Lambda,\mu_{0},\mathrm{C}_{0}),
\end{equation*}
where \(\mathrm{C}\) is a positive constant independent of \(\vec{q}\).
\end{theorem}

From this result, we derive H\"{o}lder regularity for solutions to problem \eqref{Translated problem}.

\begin{theorem}[{\bf Modulus of continuity independent of $\vec{q}$}] \label{Holderregtransl}
Let $\vec q \in \mathbb{R}^n$ and $u$ be a viscosity solution to \eqref{Translated problem} with $f, \varrho \in L^{\infty}(\Omega_1)$, $\beta, g \in C^{0,\alpha}(\partial \Omega_1)$,  and $0 < \sigma \le 1+\gamma$. The following assertions are valid:
\begin{itemize}
\item[(a)] If $0<\sigma\leq \gamma$, and 
\begin{equation*}
\max\{\|u\|_{L^{\infty}(\Omega_{r})},\|f\|_{L^{\infty}(\Omega_{r})},\|g\|_{L^{\infty}(\mathrm{T}_{r})},\|\varrho\|_{L^{\infty}(\Omega_{r})}\}\leq 1.
\end{equation*}
Then \(u\in C^{0,\bar{\alpha}}(\overline{\Omega_{\frac{r}{2}}})\) for some \(\bar{\alpha}=\bar{\alpha}(n,r,\lambda,\Lambda,\mu_{0})\in (0,1)\). Moreover,
\begin{equation*}
\|u\|_{C^{0,\bar{\alpha}}(\overline{\Omega_{\frac{r}{2}}})}\leq \mathrm{C}(n,r,\lambda,\Lambda,\mu_{0}),
\end{equation*}
where \(\mathrm{C}>0\) is independent of \(\vec{q}\).
\item[(b)] If $0<\gamma<\sigma\leq 1+\gamma$ and 
\begin{equation*}
\max\left\{\|u\|_{L^{\infty}(\Omega_{r})},\|f\|_{L^{\infty}(\Omega_{r})},\|g\|_{L^{\infty}(\mathrm{T}_{r})},\|\varrho\|_{L^{\infty}(\Omega_{r})}\left(|\vec{q}|^{\sigma-\gamma}+\Lambda\right)\right\}\leq 1.    
\end{equation*}
Then \(u\in C^{0,\bar{\alpha}}(\overline{\Omega_{\frac{r}{2}}})\) for some \(\bar{\alpha}=\bar{\alpha}(n,r,\lambda,\Lambda,\mu_{0})\in (0,1)\). Moreover,
\begin{equation*}
\|u\|_{C^{0,\bar{\alpha}}(\overline{\Omega_{\frac{r}{2}}})}\leq \mathrm{C}(n,r,\lambda,\Lambda,\mu_{0}),
\end{equation*}
where \(\mathrm{C}>0\) is independent of \(\vec{q}\).
\end{itemize}
\end{theorem}
\begin{proof}
Via scaling, we may assume that $r=1$. We will first prove item (a). We shall show that $u$ satisfies the assumptions of Theorem \ref{Holderreg} with $\theta=1$ and $\mathrm{C}_{0}=2$.  Indeed, let $\varphi\in C^{2}$ be a test function such that $u-\varphi$ attains a local maximum at $x_{0}\in \Omega_{1}$ and $|D\varphi(x_{0})-\vec{q}|>1$.Using that $u$ is, in particular, a subsolution to \eqref{Translated problem}, together with assumption \textbf{(H1)} and the smallness condition, we obtain that
\begin{eqnarray}
1&\geq& \|f\|_{L^{\infty}(\Omega_{1})}\geq f(x_{0})\geq |D\varphi(x_{0})-\vec{q}|^{\gamma}\left(F(D^{2}\varphi(x_{0}))+\frac{\varrho(x_{0})}{|D\varphi(x_{0})-\vec{q}|^{\gamma-\sigma}}\right)\nonumber\\
&\geq&|D\varphi(x_{0})-\vec{q}|^{\gamma}\left(\mathcal{M}^{-}_{\lambda,\Lambda}(D^{2}\varphi(x_{0}))-\frac{\|\varrho\|_{L^{\infty}(\Omega_{1})}}{|D\varphi(x_{0})-\vec{q}|^{\gamma-\sigma}}\right)\nonumber\\
&\geq&|D\varphi(x_{0})-\vec{q}|^{\gamma}\left(\mathcal{P}^{-}_{\lambda,\Lambda}(D^{2}\varphi(x_{0}),D\varphi(x_{0}))-\|\varrho\|_{L^{\infty}(\Omega_{1})}\right)\nonumber\\
&\geq&|D\varphi(x_{0})-\vec{q}|^{\gamma}\left(\mathcal{P}^{-}_{\lambda,\Lambda}(D^{2}\varphi(x_{0}),D\varphi(x_{0}))-1\right),\label{3.3}
\end{eqnarray}
where we used $\gamma\geq \sigma$ and $|D\varphi(x_{0})-\vec{q}|>1$. Therefore, using again that $|D\varphi(x_{0})-\vec{q}|>1$ and $\gamma>0$, it follows from \eqref{3.3} that
\begin{equation*}
\mathcal{P}^{-}_{\lambda,\Lambda}(D^{2}\varphi(x_{0}),D\varphi(x_{0}))\leq 2=\mathrm{C}_{0}.
\end{equation*}
Since $\varphi$ is arbitrary, we conclude that $u$ satisfies, in the viscosity sense,
\begin{equation}\label{3.4}
\mathcal{P}^{-}_{\lambda,\Lambda}(D^{2}u,Du)\leq \mathrm{C}_{0}
\quad \text{in}\quad
\{|Du-\vec{q}|>\theta\}\cap \Omega_{1}.
\end{equation}
On the other hand, let $\bar{\varphi}\in C^{2}$ be a test function such that
$u-\bar{\varphi}$ attains a local minimum at $x_{0}\in \Omega_{1}$ and 
$|D\bar{\varphi}(x_{0})-\vec{q}|>1$. Now, using that $u$ is a supersolution of \eqref{Translated problem}, the smallness condition, and the uniform ellipticity {\bf (H1)}, we obtain, analogously to \eqref{3.3}, that
\begin{eqnarray}
-1&\leq&-\|f\|_{L^{\infty}(\Omega_1)}\leq f(x_{0})\leq |D\bar{\varphi}(x_{0})-\vec{q}|^{\gamma}\left(F(D^{2}\bar{\varphi}(x_{0}))+\frac{\varrho(x_{0})}{|D\bar{\varphi}(x_{0})-\vec{q}|^{\gamma-\sigma}}\right)\nonumber\\
&\leq&|D\bar{\varphi}(x_{0})-\vec{q}|^{\gamma}\left(\mathcal{M}^{+}_{\lambda,\Lambda}(D^{2}\bar{\varphi}(x_{0}))+\frac{\|\varrho\|_{L^{\infty}(\Omega_{1})}}{|D\bar{\varphi}(x_{0})-\vec{q}|^{\gamma-\sigma}}\right)\nonumber\\
&\leq&|D\bar{\varphi}(x_{0})-\vec{q}|^{\gamma}\left(\mathcal{P}^{+}_{\lambda,\Lambda}(D^{2}\bar{\varphi}(x_{0}),D\bar{\varphi}(x_{0}))+\|\varrho\|_{L^{\infty}(\Omega_{1})}\right)\nonumber\\
&\leq&|D\bar{\varphi}(x_{0})-\vec{q}|^{\gamma}\left(\mathcal{P}^{+}_{\lambda,\Lambda}(D^{2}\bar{\varphi}(x_{0}),D\bar{\varphi}(x_{0}))+1\right).\label{3.5}
\end{eqnarray}
Using that $|D\bar{\varphi}(x_{0})-\vec{q}|^{\gamma}>1$ (since $|D\bar{\varphi}(x_{0})-\vec{q}|>1$ and $\gamma>0$) it follows from \eqref{3.5} that
\[
\mathcal{P}^{+}_{\lambda,\Lambda}(D^{2}\bar{\varphi}(x_{0}),D\bar{\varphi}(x_{0}))\geq -\mathrm{C}_{0}.
\]
Hence, $u$ satisfies the following inequality in the viscosity sense
\begin{equation}\label{3.6}
\mathcal{P}^{+}_{\lambda,\Lambda}(D^{2}u,Du)\geq -\mathrm{C}_{0}
\quad \text{in}\quad
\{|Du-\vec{q}|>\theta\}\cap \Omega_{1}.
\end{equation}
Therefore, together with the boundary equation on $\partial\Omega_{1}$ and the smallness assumptions on the data, \eqref{3.4} and \eqref{3.6} imply that $u$ is a viscosity solution to the system \eqref{3.2} with the stated constants. The H\"{o}lder regularity now follows from Theorem \ref{Holderreg}.

For item (b), given a test function $\varphi \in C^{2}$ such that $u-\varphi$ attains a local maximum at some $x_{0}\in \Omega_{1}$ and $|D\varphi(x_{0})-\vec{q}|>1$ using that $u$ is a subsolution of \eqref{Translated problem}, it follows that
\begin{eqnarray}
1\geq \|f\|_{L^{\infty}(\Omega_{1})}\geq  |D\varphi(x_{0})-\vec{q}|^{\gamma}\left(\mathcal{M}^{-}_{\lambda,\Lambda}(D^{2}\varphi(x_{0}))-\|\varrho\|_{L^{\infty}(\Omega_{1})}|D\varphi(x_{0})-\vec{q}|^{\sigma-\gamma}\right). \label{3.7}
\end{eqnarray}
Now, let us split the analysis into two cases:\\
I. $|D\varphi(x_{0})|\geq 1$\\
Note that, using this fact and  $\sigma-\gamma\in (0,1]$ in \eqref{3.7}, we have that
\begin{eqnarray*}
1&\geq&|D\varphi(x_{0})-\vec{q}|^{\gamma}\left(\mathcal{M}^{-}_{\lambda,\Lambda}(D^{2}\varphi(x_{0}))-\|\varrho\|_{L^{\infty}(\Omega_{1})}(|D\varphi(x_{0})|^{\sigma-\gamma}+|\vec{q}|^{\sigma-\gamma}\right)\nonumber\\
&\geq&|D\varphi(x_{0})-\vec{q}|^{\gamma}\left(\mathcal{M}^{-}_{\lambda,\Lambda}(D^{2}\varphi(x_{0}))-\|\varrho\|_{L^{\infty}(\Omega_{1})}(|D\varphi(x_{0})|+|\vec{q}|^{\sigma-\gamma})\right)\nonumber\\
&\geq&|D\varphi(x_{0})-\vec{q}|^{\gamma}\left(\mathcal{P}^{-}_{\lambda,\Lambda}(D^{2}\varphi(x_{0}),D\varphi(x_{0}))+|D\varphi(x_{0})|(\Lambda-\|\varrho\|_{L^{\infty}(\Omega_{1})})-1\right)\nonumber\\
&\geq&|D\varphi(x_{0})-\vec{q}|^{\gamma}\left(\mathcal{P}^{-}_{\lambda,\Lambda}(D^{2}\varphi(x_{0}),D\varphi(x_{0}))-1\right),
\end{eqnarray*}
since by the smallness condition we have that $\|\varrho\|_{L^{\infty}(\Omega_{1})}\leq \Lambda$ and $\|\varrho\|_{L^{\infty}(\Omega_{1})}|\vec{q}|^{\sigma-\gamma}\leq 1$.\\
II, $|D\varphi(x_{0})|<1$\\
In this case, using that $\|\varrho\|_{L^{\infty}(\Omega_{1})}\leq \Lambda$ and $\|\varrho\|_{L^{\infty}(\Omega_{1})}|\vec{q}|^{\sigma-\gamma}\leq 1$, we can conclude that 
\begin{eqnarray*}
1&\geq&|D\varphi(x_{0})-\vec{q}|^{\gamma}\left(\mathcal{M}^{-}_{\lambda,\Lambda}(D^{2}\varphi(x_{0}))-\|\varrho\|_{L^{\infty}(\Omega_{1})}(|D\varphi(x_{0})|^{\sigma-\gamma}+|\vec{q}|^{\sigma-\gamma}\right)\nonumber\\
&\geq&|D\varphi(x_{0})-\vec{q}|^{\gamma}\left(\mathcal{M}^{-}_{\lambda,\Lambda}(D^{2}\varphi(x_{0}))-\left(1+\Lambda\right)\right)\nonumber\\
&\geq&|D\varphi(x_{0})-\vec{q}|^{\gamma}\left(\mathcal{P}^{-}_{\lambda,\Lambda}(D^{2}\varphi(x_{0}),D\varphi(x_{0}))-\left(1+\Lambda\right)\right).
\end{eqnarray*}
In this case, by I. and II., it follows that \eqref{3.5} is also valid with $\mathrm{C}_{0}=\max\left\{2+\Lambda,2\right\}$. Similarly, one can show that $u$ satisfies
\begin{equation*}
\mathcal{P}^{+}_{\lambda,\Lambda}(D^{2}u,Du)\geq -\mathrm{C}_{0}\,\,\, \text{in}\,\,\, \{|Du-\vec{q}|>1\}\cap \Omega_{1}.
\end{equation*}
Therefore, we can conclude that $u$ solves the system \eqref{3.2} for  these constants $\mathrm{C}_{0}$ and $\theta$, and again, by applying Theorem \ref{Holderreg}, it follows that $u$ is H\"{o}lder continuous.
\end{proof}

\section{Proof of the main results}\label{Section4}

In this section, we prove Theorem \ref{main} and Corollary \ref{main1}. Our approach relies on the Geometric Tangential Method combined with an improvement-of-flatness scheme, which yields the desired boundary $C^{1,\alpha}$ estimates.

\subsection{An Approximation result}

Using the compactness result established in the previous section, we show that solutions to
the translated problem \eqref{Translated problem} are stable under uniform limits and can be approximated by
solutions to a homogeneous uniformly elliptic equation with constant oblique boundary data. 

To this end, we note that problem \eqref{Translated problem} is posed in a
domain $\Omega_1$ that varies according to the boundary
transformation $\psi$. Thus, we need to compare solutions defined on these varying domains with functions defined on half-balls. To make such comparisons precise, following \cite[Subsection 3.1]{BesOh26}, we adopt the following extension convention.

For $0<\tilde{\alpha},\rho\leq 1$, whenever $u\in C^{0,\tilde{\alpha}}(\overline{\Omega_{\rho}})$, we denote by
$\overline{u}$ the McShane extension (see \cite{McS34}) to $\mathbb{R}^n$ defined by
\[
\overline{u}(x):=\max\Big\{-\|u\|_{L^{\infty}(\Omega_{\rho})},\ \min\Big\{\|u\|_{L^{\infty}(\Omega_{\rho})},\ \inf_{z\in\overline{\Omega_{\rho}}}\big(u(z)+[u]_{C^{0,\tilde{\alpha}}(\overline{\Omega_{\rho}})}|x-z|^{\tilde{\alpha}}\big)\Big\}\Big\}.
\]
In this setting, $\bar{u}$ has the same $L^{\infty}$ norm and the same $\tilde{\alpha}$-H\"{o}lder seminorm as $u$. Likewise, any function
$v\in C^{1,\tilde{\alpha}}(\overline{\mathrm{B}_{r}^{+}})$ may be identified with a $C^{1,\tilde{\alpha}}$ extension to $\mathbb{R}^{n}$ such that
$$
\|v\|_{C^{1,\tilde{\alpha}}(\mathbb{R}^{n})}\leq \mathrm{C}(n,\tilde{\alpha})\,\|v\|_{C^{1,\tilde{\alpha}}(\overline{\mathrm{B}^{+}_{\rho}})}.
$$
The existence of such an extension follows from Whitney's extension theorem \cite{Whi34}, since $\overline{\mathrm{B}^{+}_{\rho}}$ is a convex domain.

With this convention, comparisons between functions defined on different domains are understood in terms of their extensions. In particular, in the compactness argument below, the boundary values prescribed on $\partial \mathrm{B}_{\rho}^{+}\setminus \mathrm{T}_{\rho}$ are understood as the restrictions of the corresponding McShane extensions of the solutions $u$. This convention allows us to formulate the comparison problems on fixed half-balls, despite the dependence of the original domains on the boundary transformations.

With this observations in mind, we can enunciate the following result.
\begin{theorem}[\bf Approximation tool] \label{stab}
Let $\vec q\in\mathbb{R}^n$ and let $u$ be a normalized\footnote{By a normalized solution, we mean a solution satisfying $\|u\|_{L^\infty(\Omega_1)}\le 1$} viscosity solution of \eqref{Translated problem}. Given $\delta>0$, there exists $\epsilon=\epsilon(\delta,n,\lambda,\Lambda,\gamma,\mu_0)>0$ such that if
$$
\|f\|_{L^{\infty}(\Omega_1)}+ \|g\|_{L^{\infty}(\mathrm{T}
_1)}+ \|\beta-\beta_0\|_{C^{0,\alpha}(\mathrm{T}
_1)}+\|\psi_{\Omega}\|_{C^1(\mathrm{T}
_1)}+\|\varrho\|_{L^{\infty}(\Omega_1)}(|\vec q|^{(\sigma-\gamma)_+ }+1) < \epsilon,
 $$
then there exists a normalized solution
$v\in C^{1,\alpha_0}(\mathrm{B}^+_{3/4})$ to
 $$
\left\{
\begin{array}{rclcl}
F(D^2 v)  &=& 0 & \mbox{in} & \mathrm{B}^{+}_{\frac{3}{4}}, \\
\beta_0 \cdot Dv&=& 0 & \mbox{on} & \mathrm{T}_{\frac{3}{4}},
 \end{array}
    \right.
 $$
where $\beta_0:=\beta(0)$  and
 $$
 \|u-v\|_{L^{\infty}(\Omega_{1/2})} \le \delta. 
 $$
\end{theorem}

\begin{proof}
We argue by contradiction. Suppose the conclusion fails. Then there exist $\delta_0>0$
and sequences $F_k$, $u_k$, $f_k$, $\varrho_k$, $g_k$, $\beta_k$, $\Omega_k$, and $\vec q_k$
such that $\|u_k\|_{L^\infty((\Omega_k)_1)} \le 1$ and
\begin{equation*}
\left\{
\begin{array}{rclcl}
|Du_k-\vec{q}_k|^{\gamma} F_k(D^2 u_k) + \varrho_k(x) |Du_k-\vec{q}_k|^{\sigma} &=& f_k(x) & \mbox{in} & (\Omega_k)_1, \\
\beta_k \cdot D u_k &=& g_k(x) & \mbox{on} & \partial (\Omega_k)_1
    \end{array}
    \right.
 \end{equation*}
 satisfying 
 \begin{equation} \label{4}
 \|f_k\|_{L^{\infty}((\Omega_{k})_{1})}+ \|g_k\|_{L^{\infty}(\mathrm{T}_1)}+ \|\beta_k-\beta_k(0)\|_{C^{0,\alpha}(\mathrm{T}_{1})}+ \|\psi_{\Omega_k}\|_{C^1(\mathrm{T}_1)}+ \|\varrho_k\|_{L^{\infty}((\Omega_{k})_{1})}(|\vec q_k|^{(\sigma-\gamma)_+}+1) \le \frac{1}{k}
 \end{equation}
 but such that
 $$
 \|u_k-v\|_{L^{\infty}((\Omega_{k})_{1/2})} \ge \delta_0
 $$
 for every normalized viscosity solution $v\in C^{1,\alpha_{0}}(\mathrm{B}^{+}_{3/4})$ to
 $$
\left\{
\begin{array}{rclcl}
F_{k}(D^2 v)  &=& 0 & \mbox{in} & \mathrm{B}^{+}_{\frac{3}{4}}, \\
(\beta_{k})_0 \cdot Dv&=& 0 & \mbox{on} & \mathrm{T}_{\frac{3}{4}}.
 \end{array}
 \right.
 $$ 
  Notice that
 $$
\|\varrho_k\|_{L^{\infty}} \cdot \left(|\vec q_k|^{(\sigma-\gamma)_{+}} +\max\left\{1,\Lambda\right\}\right) \le \frac{1+\max\left\{1,\Lambda\right\}}{k}
 $$
 for all $k \ge 1$. Hence, by Theorem \ref{Holderregtransl}, the sequence $(u_k)_{k\in\mathbb N}$ is precompact in $C^{0}_{\mathrm{loc}}$, precisely, the sequence is uniformly bounded $C^{0, \tilde{\alpha}}(\overline{(\Omega_{k})_{4/5}})$ for some $\tilde{\alpha}=\tilde{\alpha}(n,\lambda,\Lambda,\mu_{0})\in (0,1)$. Thus, by the previous comment, we can consider the extension convention, so the functions $u_{k}$ belong to $C^{0, \tilde{\alpha}}(\overline{\mathrm{B}^{+}_{4/5}})$ and the sequence is uniformly bounded. In this case, by Arzelà-Ascoli Theorem there exists a $u_{\infty}\in C^{0, \tilde{\alpha}}(\overline{\mathrm{B}^{+}_{4/5}})$ such that, up to a subsequence, $u_{k}$ converges uniformly to $u_{\infty}$ in $C^{0, \tilde{\alpha}}(\overline{\mathrm{B}^{+}_{4/5}})$. Note that the restriction of $u_{\infty}$ to
$\overline{\mathrm{B}^{+}_{\frac{4}{5}}}$, which is the only one used below, is independent of the extensions chosen. Indeed, every point of $\mathrm{B}^{+}_{\frac{4}{5}}$ belongs to $(\Omega_k)_1$ for all sufficiently large $k$, while every point of $\mathrm{T}_{\frac{4}{5}}$ lies at a distance less than $\frac{1}{k}$ from $\partial(\Omega_k)_1$, owing to the convergence $\psi_{\Omega_k}\to 0$.  In addition, by the uniform ellipticity of $F_k$ and condition \eqref{4}, we may assume (up to a subsequence) that $F_k \to F_\infty$ locally uniformly in $\mathrm{Sym}(n)$. Moreover, $\beta_k \to \beta_\infty$ uniformly on $\mathrm{T}_{\frac{4}{5}}$ for some constant oblique vector $\beta_\infty$, and $\Omega_k\to \mathrm{B}^+_{\frac{4}{5}}$ in the sense of domains.
 
 We claim that the limit $u_\infty$ is a viscosity solution of
 \begin{equation}\label{6}
  \left\{
    \begin{array}{rclcl}
     F_{\infty}(D^2 u_{\infty}) &=& 0 & \mbox{in} & \mathrm{B}^+_\frac{4}{5}, \\
    \beta_{\infty} \cdot D u_{\infty} &=& 0 & \mbox{on} & \mathrm{T}_\frac{4}{5}.
    \end{array}
    \right.
 \end{equation}
Indeed, if $(\vec q_k)_{k\in\mathbb N}$ is bounded, then (up to a subsequence) $\vec q_k\to \vec q_\infty$,
and by stability $u_\infty$ solves
$$
  \left\{
    \begin{array}{rclcl}
     |Du_{\infty}-\vec q_{\infty}|^{\gamma}F_{\infty}(D^2 u_{\infty}) &=& 0 & \mbox{in} & \mathrm{B}^+_\frac{4}{5}, \\
    \beta_{\infty} \cdot D u_{\infty} &=& 0 & \mbox{on} & \mathrm{T}_\frac{4}{5}.
    \end{array}
    \right.
 $$
 By the cutting lemma (cf. \cite[Lemma 2.5]{BesRicSil26}), it follows that $u_\infty$ solves \eqref{6}.

On the other hand, if $(\vec q_k)_{k\in\mathbb N}$ is unbounded then, up to a subsequence, we may assume that $|\vec q_k|>k$ for every $k\in\mathbb N$. The boundary condition passes to the limit (via stability), yielding
$\beta_\infty\cdot Du_\infty = 0$ on $\mathrm{T}_{\frac{4}{5}}$. It remains to show that $u_\infty$ satisfies
$F_\infty(D^2u_\infty)=0$ in $\mathrm{B}^+_{\frac{4}{5}}$.
 
We prove that $u_\infty$ is a viscosity subsolution, since the supersolution case is analogous. Fix $x_0\in \mathrm{B}^+_{\frac{4}{5}}$ and consider the quadratic polynomial 
\[
p(x):=u_\infty(x_{0})+\vec{b}_{0}\cdot(x-x_{0})+\frac{1}{2}(x-x_{0})^T \mathrm{M}_{0}(x-x_{0}).
\]
Assume that $p$ touches $u_\infty$ from above at $x_0$ in $\mathrm{B}^+_{\frac{4}{5}}$. We aim to show that $F_\infty(\mathrm{M}_0)\le 0$. For this, fix $0<r\ll 1$ such that $\overline{\mathrm{B}_{r}(x_{0})}\subset \mathrm{B}^{+}_{\frac{4}{5}}$ and define the points $x_k\in \overline{\mathrm{B}_{r}^{+}(x_{0})}$ by
\[
p(x_k)-u_k(x_k)=\max_{x\in \overline{\mathrm{B}_{r}(x_{0})}}\big(p(x)-u_k(x)\big).
\]
Moreover, observe that from the unboundedness of the sequence $(\vec{q}_{k})$ it follows that there exists $k^{\star} \in \mathbb{N}$ such that
\begin{equation}\label{conddeb}
\max \{1,|\vec b_{0}|+r\|\mathrm{M}_{0}\|\} < |\vec q_k|\quad \text{for all}\quad k\geq k^{\star}.
\end{equation}
Since $u_k$ is a viscosity subsolution, we obtain
\begin{equation}\label{eqdesub}
\left | \frac{\vec b_{0}+\mathrm{M}_{0}(x_{k}-x_{0})}{|\vec q_k|} + \vec e_k \right |^{\gamma}F_k(\mathrm{M}_{0}) + \tilde{H}_k(x_{k},\vec b_{0} + \mathrm{M}_{0}(x_{k}-x_{0})) \leq \frac{f_k(x_{k})}{|\vec q_k|^{\gamma}},
\end{equation}
where
$$
\tilde{H}_k(x,\vec{\xi}) := \varrho_{k}(x)|\vec q_k|^{-\gamma}|\vec{\xi}-\vec{q}_{k}|^{\sigma}\qquad \text{and}\qquad \vec{e}_{k}=-\frac{\vec{q}_{k}}{|\vec{q}_{k}|}.
$$ 
Observe that the term in \eqref{eqdesub} associated with $\tilde{H}_{k}$ satisfies
\begin{eqnarray}
|\tilde{H}_k(x_{k},\vec{b}_{0}+\mathrm{M}_{0}(x_{k}-x_{0}))| &\le& \|\varrho_k\|_{L^{\infty}((\Omega_{k})_{1})} |\vec q_k|^{\sigma-\gamma} \left | \frac{\vec b_{0}+\mathrm{M}_{0}(x_{k}-x_{0})}{|\vec q_k|} + \vec e_k \right |^{\sigma}\nonumber\\
&\le& \|\varrho_k\|_{L^{\infty}((\Omega_{k})_{1})} |\vec q_k|^{\sigma-\gamma} \left( \frac{|\vec{b}_{0}+\mathrm{M}_{0}(x_{k}-x_{0})|}{|\vec q_k|} +1 \right)^\sigma\nonumber\\
&\le& \|\varrho_k\|_{L^{\infty}((\Omega_{k})_{1})} |\vec q_k|^{\sigma-\gamma} \left( \frac{|\vec{b}_{0}|+r\|\mathrm{M}_{0}\|}{|\vec q_k|} +1 \right)^\sigma\nonumber\\
&\le& 2^\sigma \frac{|\vec q_k|^{\sigma-\gamma}}{k (|\vec q_k|^{(\sigma-\gamma)_{+}}+1)},\label{conddeH}	
\end{eqnarray}
where we used \eqref{4} and \eqref{conddeb}. Now, let us analyze the right-hand side of estimate \eqref{conddeH}:
\begin{enumerate}
\item[(i)] If $0< \sigma-\gamma$, then
$$
2^\sigma \frac{|\vec q_k|^{\sigma-\gamma}}{k (|\vec q_k|^{(\sigma-\gamma)_{+}}+1)} \le \frac{2^\sigma}{k} \to 0, \,\,\, \textrm{as} \,\,\, k \to +\infty.
$$
\item[(ii)] If $\sigma \le \gamma$, then $(\sigma-\gamma)_+=0$. Thus,  since $\sigma-\gamma \le 0$ and $|\vec q_k| \to \infty$, the term $|\vec q_k|^{\sigma-\gamma} \le 1$  is bounded. Therefore, 
$$
2^\sigma \frac{|\vec q_k|^{\sigma-\gamma}}{k (|\vec q_k|^{(\sigma-\gamma)_{+}}+1)}  \le \frac{\mathrm{C}}{k} \to 0.
$$
\end{enumerate}
Hence, from cases (i) and (ii), we can conclude that 
$\tilde{H}_k(x_{k},\vec{b}_{0}+\vec{q}_{k}) \to 0$ as $k\to \infty$. Moreover, by the condition \eqref{4} it follows 
\[
\left|\frac{f_{k}(x_{k})}{|\vec{q}_{k}|^{\gamma}}\right|\leq \frac{1}{k|\vec{q}_{k}|^{\gamma}}\to 0\,\,\, \text{as} \,\,\, k\to \infty.
\]
Finally, as $|\vec e_{\infty}|=1$, up to a subsequence, $\vec e_k \to \vec e_{\infty}$ for some unit vector $\vec{e}_{\infty}$. Therefore, in view of these convergences, letting $k\to\infty$ in \eqref{eqdesub} yields
$$
|0+\vec e_{\infty}|^{\gamma} F_{\infty}(\mathrm{M}_{0})\leq0 .
$$
By the cutting lemma as in \cite{Silvestre}, we have $F_{\infty}(\mathrm{M}_{0})\leq0$. This proves the claim. Hence, $u_{\infty}$ satisfies \eqref{6}.

Now, let $w_k$ be the unique viscosity solution of
$$
 \left\{
    \begin{array}{rclcl}
     F_{k}(D^2 w_{k}) &=& 0 & \mbox{in} & \mathrm{B}^+_{\frac{3}{4}}, \\
    \beta_{k}(0) \cdot D w_{k} &=& 0 & \mbox{on} & \mathrm{T}_{\frac{3}{4}},\\
    w_k &=& u_k & \mbox{on} & \partial \mathrm{B}^{+}_{\frac{3}{4}}\setminus \mathrm{T}_{\frac{3}{4}}.
    \end{array}
    \right.
$$
Existence and uniqueness follow from \cite[Theorem 3.3]{LiZhang}. Moreover, by Alexandrov-Bakelman-Pucci estimate for oblique boundary value problems \cite[Theorem 2.1]{LiZhang}, together with the fact that the extensions $u_k$ are bounded by $1$ in the $L^\infty$ norm, we obtain $\|w_k\|_{L^\infty(\mathrm{B}^{+}_{\frac{3}{4}})}\leq 1$. By stability, $w_k\to w_\infty$ locally uniformly, where $w_\infty$ satisfies
$$
  \left\{
    \begin{array}{rclcl}
     F_{\infty}(D^2 w_{\infty}) &=& 0 & \mbox{in} & \mathrm{B}^+_{\frac{3}{4}}, \\
     \beta_0 \cdot D w_{\infty} &=& 0 &\mbox{on}& \mathrm{T}_{\frac{3}{4}}, \\
    w_{\infty} &=& u_{\infty} & \mbox{on} & \partial \mathrm{B}^{+}_{\frac{3}{4}}\setminus \mathrm{T}_{\frac{3}{4}}.
    \end{array}
    \right.
 $$
By uniqueness, we obtain $w_\infty=u_\infty$. Finally, note that the extensions of $u_k$ and
$w_k$ are uniformly Hölder continuous. For the latter, this
follows from the uniform $C^{1,\alpha_0}$ bound on
$\overline{\mathrm{B}^{+}_{\frac{2}{3}}}$, which is provided
by the structural conditions {\bf (H1)}--{\bf (H2)} for
solutions bounded by $1$, as in the proof of
Lemma~\ref{Holderregtransl}. Since
$\overline{(\Omega_k)_{\frac{1}{2}}}$ lies within distance
$\frac{1}{k}$ of $\overline{\mathrm{B}^{+}_{\frac{1}{2}}}$,
it follows that the convergence
$u_k-w_k\to 0$ holds uniformly on
$\overline{(\Omega_k)_{\frac{1}{2}}}$. This  contradicts the assumption
$\|u_k-v\|_{L^\infty((\Omega_k)_{1/2})}\ge \delta_0$ for $k$ large, which completes the proof.
\end{proof}

\subsection{Proof of Theorem \ref{main}} 

We begin by establishing a discrete improvement-of-flatness estimate, which is the key step towards the optimal boundary $C^{1,\alpha}$ regularity.

\begin{lemma}[\textbf{Improvement of flatness}]\label{key 1}
Let $u$ be a normalized viscosity solution of
\[
\begin{cases}
|Du-\vec q|^\gamma F(D^2u)+\varrho(x)|Du-\vec q|^\sigma = f(x) & \text{in } \Omega_1,\\
\beta\cdot Du = g(x) & \text{on } \partial\Omega_1,
\end{cases}
\]
for some $\vec q\in\mathbb R^n$. Fix $0<\hat\alpha<\alpha_0$. Then there exist $0<r<1/2$ and $\epsilon>0$ such that if
\[
\|f\|_{L^\infty(\Omega_1)}+\|g\|_{L^\infty(\mathrm{T}_1)}
+\|\beta-\beta_0\|_{C^{0,\alpha}(\mathrm{T}_1)}
+\|\psi\|_{C^1(\mathrm{T}_1)}
+\|\varrho\|_{L^\infty(\Omega_1)}\big(|\vec q|^{(\sigma-\gamma)_+}+1\big)
<\epsilon,
\]
then there exists an affine function $\ell(x)=u(0)+\vec b\cdot x$ such that
\[
\|u-\ell\|_{L^\infty(\Omega_r)}\le r^{1+\hat\alpha},
\qquad
\beta_0\cdot \vec b = 0,
\qquad
|\vec b|\le \mathrm{C}_\star.
\]
Furthermore, when $\gamma<\sigma\leq 1+\gamma$, the constant $r$ additionally satisfies the condition 
$$
r^{1-\hat{\alpha}}\left(\frac{\mathrm{C}_{\star}}{1-r^{\hat{\alpha}}}\right)^{\sigma-\gamma}<1.
$$
\end{lemma}

\begin{proof}
Fix $\delta>0$, and let $\epsilon>0$ be given by Theorem 4.1. Then there exists
$v\in C^{1,\alpha_0}(\mathrm{B}^+_{3/4})$ be a normalized solution to
\[
\begin{cases}
F(D^2v)=0 & \text{in } \mathrm{B}^+_{\frac{3}{4}},\\
\beta_0\cdot Dv=0 & \text{on } \mathrm{T}_{\frac{3}{4}},
\end{cases}
\]
and
\[
\|u-v\|_{L^\infty(\Omega_{1/2})}\le \delta.
\]
Let $\tilde \ell(x) := v(0)+Dv(0)\cdot x$. By the $C^{1,\alpha_0}$ estimates in {\bf(H2)}, we have
\[
|v(x)-\tilde \ell(x)| \le \mathrm{C}_\star |x|^{1+\alpha_0}
\quad \text{and} \quad
\beta_0\cdot Dv(0)=0,
\quad
|Dv(0)|\le \mathrm{C}_\star.
\]
Choose $0<r<1/2$ such that
\[
\mathrm{C}_\star r^{1+\alpha_0} \le \frac{1}{3} r^{1+\hat{\alpha}}, (\text{and such that } \ \mathrm{C}_{\star}r^{\frac{1-\hat{\alpha}}{\sigma-\gamma}}+r^{\hat{\alpha}}<1 \ \text{when $\gamma<\sigma\leq 1+\gamma$})
\]
and set $\delta:=\frac{1}{3} r^{1+\hat{\alpha}}$. Defining $\ell(x):=u(0)+Dv(0)\cdot x$ we obtain, for every
$x\in \Omega_r$,
\[
|u(x)-\ell(x)|
\le |u(x)-v(x)| + |v(x)-\tilde \ell(x)| + |v(0)-u(0)|.
\]
Therefore,
\[
\|u-\ell\|_{L^\infty(\Omega_r)}
\le 2\|u-v\|_{L^\infty(\Omega_r)} + \|v-\tilde \ell\|_{L^\infty(\Omega_r)}
\le 2\delta + \mathrm{C}_\star r^{1+\alpha_0}
\le r^{1+\hat{\alpha}}.
\]
Finally, setting $\vec b:=Dv(0)$, we have $\beta_0\cdot \vec b=0$ and $|\vec b|\le \mathrm{C}_\star$.
This completes the proof.
\end{proof}

Now we are in a position to prove Theorem \ref{main}.

\begin{proof}[\bf Proof of Theorem \ref{main}] 
If $u\equiv 0$, the conclusion is trivial. Otherwise, the strategy is to
construct a sequence of affine functions $\ell_k(x)=a_k+\vec{\mathfrak{b}}_k\cdot x$ such that
\begin{itemize}
\item [(i)] $  |a_{k+1}-a_k| \le \mathrm{C}_{\star} r^{(1+\alpha^{\prime})k}, \,\, \|\vec{\mathfrak{b}}_{k+1} - \vec{\mathfrak{b}}_{k}\| \leq  \mathrm{C}_{\star}r^{k\alpha^{\prime}}$,
\item [(ii)] $\beta_0 \cdot \vec b_k =0$,
\item [(iii)] $\|u - \ell_{k}\|_{L^{\infty}(\Omega_{r^k})} \leq r^{k(1+\alpha^{\prime})}$,
\end{itemize}
where $0<r<1/2$ is given by Lemma \ref{key 1} and
\[
\alpha'=\min\left\{\alpha_0^-,\alpha,\frac{1}{1+\gamma}\right\}.
\]
The proof is carried out in three steps. We first reduce the problem to a normalized framework. Then, we build an iterative affine approximation scheme based on Lemma \ref{key 1}. Finally, we conclude the desired $C^{1,\alpha'}$ estimate by a standard compactness and covering
argument.
\begin{enumerate}
\item[\textbf{Step 1.}] \textbf{Reduction of the problem}\\
Let $r,\epsilon >0$ be as in Lemma \ref{key 1}. By scaling and normalization, we may assume that $\|u\|_{L^\infty(\Omega_1)}\le 1$, $u(0)=0$, and that $u$ satisfies
$$
\left\{
\begin{array}{rclcl}
\mathcal{G}(D^{2}u,Du-\vec{q},x) &=& f(x) & \mbox{in} & \Omega_1, \\
      && \\
\mathcal{B}(Du,0,x)  &=& g(x) & \mbox{on} & \partial \Omega_1,
    \end{array}
    \right.
$$	
for some $\vec q\in\mathbb R^n$, with $g(0)=0$, $\psi(0)=0$, and $D\psi(0)=0$, and such that
 $$
\|f\|_{L^{\infty}(\Omega_{1})} + \|g\|_{C^{0,\alpha}(\mathrm{T}_{1})} + \|\beta -\beta_0\|_{C^{0,\alpha}(\mathrm{T}_{1})} + \|\varrho\|_{L^{\infty}(\Omega_{1})} (|\vec q|^{(\sigma-\gamma)_+}+1) + \|\psi\|_{C^1(\mathrm{T}_{1})}\le \epsilon.
$$
In fact, consider the rescaled function $\tilde{u}(x):=\frac{u(sx)}{K}$, with $s\in(0,1)$ and $K>0$, where $s$ and $K$ will be chosen appropriately later. A direct computation shows that $\tilde{u}$ satisfies
$$
 \left\{
    \begin{array}{rclcl}
\tilde{\mathcal{G}}(D^{2}\tilde{u},D\tilde{u},x) &=& \tilde{f}(x) & \mbox{in} & \left(\frac{1}{s}\Omega\right)\cap \mathrm{B}_1, \\
      && \\
\tilde{\beta}\cdot D\tilde{u} &=& g_1(x) & \mbox{on} & \partial \left( \frac{1}{s}\Omega \right) \cap \mathrm{B}_1.
    \end{array}
    \right.
$$	
where
\[
\left\{
\begin{aligned}
\tilde{\mathcal{G}}(\mathrm{M},\vec{\xi},x)&:= |\vec\xi|^{\gamma}\frac{s^2}{K} F\left(\frac{K}{s^2}\mathrm{M}\right)+\tilde{\varrho}(x)|\vec\xi|^{\sigma},\\
\tilde{\varrho}(x) &= \frac{s^{2+\gamma-\sigma}}{K^{1+\gamma-\sigma}}\varrho(sx),\\
\tilde{f}(x) &= \frac{s^{2+\gamma}}{K^{1+\gamma}}f(sx),\\
\tilde{\beta}(x) &:= \beta(sx),\\
g_1(x)&= \frac{s}{K}g(sx).
\end{aligned}
\right.
\]
Now, fix $\nu_{\star}>0$, to be chosen sufficiently small. After a suitable change of
coordinates, we may assume that $\Omega_1$ is locally represented as the graph of a $C^1$
function $\psi_{\Omega}$ satisfying $\psi_{\Omega}(0)=0$ and $D\psi_{\Omega}(0)=0$.
Since $\psi_{\Omega}\in C^1$, there exists $0<\tilde{s}\ll1$, depending only on the $C^1$ modulus
of $\psi_{\Omega}$, such that
\[
|D\psi_{\Omega}(x)|\le \nu_{\star}
\qquad\text{for all } |x|\le \tilde{s}.
\]
Set
\[
\tilde{\psi}(x):=\psi_{\frac{1}{\tilde{s}}\Omega}(x)=\frac{\psi(\tilde{s}x)}{\tilde{s}},
\]
for $s\leq \tilde{s}$. Then, $\|\tilde{\psi}\|_{C^1(\mathrm{T}_{s^{-1}})}\le \nu_{\star}$ and
\[
\|\tilde{\beta}-\beta_0\|_{C^{0,\alpha}(\mathrm{T}_{s^{-1}})}
\le s^{\alpha}\|\beta-\beta_0\|_{C^{0,\alpha}(\mathrm{T}_{1})}.
\]
Accordingly, we choose $s\in(0,1)$ such that $s\le \tilde{s}$,
\[
\|\tilde{\psi}\|_{C^1(\mathrm{T}_{s^{-1}})}\le \nu_{\star}\le \frac{\epsilon}{5} \qquad\text{and}\qquad
\|\tilde{\beta}-\beta_0\|_{C^{0,\alpha}(\mathrm{T}_{s^{-1}})}
\le \frac{1-r^{\alpha}}{2\mathrm{C}_{\star}}\nu_{\star}\le \frac{\epsilon}{5},
\]
where $\epsilon$ and $0<r<\frac{1}{2}$ are the constants given in Lemma \ref{key 1}. Thus,
\begin{itemize}
\item For $0<\sigma<1+\gamma$, we set
$$
K \colon= 4 \left(\|u\|_{L^{\infty}(\Omega_{1})} + (4\nu_{\star}^{-1}\|\varrho\|_{L^{\infty}(\Omega_{1})})^{\frac{1}{1+\gamma-\sigma}}+\mu^{-1}_0 \nu_{\star}^{-1}\left(\|f\|^{\frac{1}{1+\gamma}}_{L^{\infty}(\Omega_{1})}+ \|g\|_{C^{0,\alpha}(\mathrm{T
}_{1})}\right) \right).
$$	
\item In the linear situation $\sigma=1+\gamma$, we choose $s\in(0,1)$ fulfilling the conditions stated above and, in addition, such that 
\[
s\leq \left(\frac{\nu_{\star}}{2\|\varrho\|_{L^{\infty}(\Omega_{1})}}\right)^{\frac{1}{2+\gamma-\sigma}},
\]
whenever $\varrho\neq 0$ (in the remaining case, the optimal $C^{1,\alpha}$ estimates are already available; see, for instance, \cite{BesOh26} and \cite{BesRicSil26}). Furthermore, we set
$$
K \colon= 4 \left(\|u\|_{L^{\infty}(\Omega_{1})} + \mu^{-1}_0 \nu_{\star}^{-1}\left(\|f\|^{\frac{1}{1+\gamma}}_{L^{\infty}(\Omega_{1})}+ \|g\|_{C^{0,\alpha}(\mathrm{T
}_{1})}\right)\right).
$$	
\end{itemize}
With these choices, we have $\|\tilde{u}\|_{L^{\infty}((s^{-1}\Omega)_{1})} \le \frac{1}{4}$,
$\|\tilde{f}\|_{L^{\infty}((s^{-1}\Omega)_{1})} \le \nu_{\star}\le \frac{\epsilon}{5}$,
$\|\tilde{\varrho}\|_{L^{\infty}((s^{-1}\Omega)_{1})}\leq \frac{\nu_{\star}}{4}\leq \frac{\epsilon}{10}$,
and
\[
\|g_1\|_{C^{0,\alpha}(\mathrm{T}_{s^{-1}})} \le \frac{\mu_0}{4} \nu_{\star} \le \frac{\epsilon}{5}.
\]
Now, define $\tilde{v}(x)= \tilde{u}(x)-\frac{\tilde{g}(0)}{\tilde{\beta}_n(0)} x_n -\tilde{u}(0)$
and $\tilde{q} =-\frac{g_1(0)}{\tilde{\beta}_n(0)} e_n$. Then $\tilde{v}$ satisfies
$$
\left\{
\begin{array}{rclcl}
\tilde{G}(D^{2}v,D\tilde{v}- \tilde{q},x)&=& \tilde{f}(x) & \mbox{in} & \tilde{\Omega}_1 =\left(\frac{1}{s}\Omega\right)\cap \mathrm{B}_1, \\
      && \\
\tilde{\mathcal{B}}(Dv,0,x)&=& \tilde{g}(x) & \mbox{on} & \partial \tilde{\Omega_1}=\partial \left( \frac{1}{s}\Omega \right) \cap \mathrm{B}_1,
    \end{array}
    \right.
$$	
where $\tilde{\mathcal{B}}(\vec{\xi},0,x)=\beta(x)\cdot \vec{\xi}$ and $\tilde{g}=g_1-\frac{g_1(0)}{\tilde{\beta}_n(0)} \tilde{\beta}_n$. Note that
$\tilde{v}(0)=0$, $\tilde{g}(0)=0$, and
\[
\|\tilde{v}\|_{L^{\infty}((\Omega_{s^{-1}})_{1})}
\le 2 \|\tilde{u}\|_{L^{\infty}((\Omega_{s^{-1}})_{1})} +\frac{|g_1(0)|}{\mu_0}
\le 1.
\]
Moreover,
\begin{eqnarray*}
\|\tilde{g}\|_{C^{0,\alpha}(\mathrm{T}_{s^{-1}})}
&\le&
\|g_1\|_{C^{0,\alpha}(\mathrm{T}_{s^{-1}})}
+\frac{|g_1(0)|}{\mu_0}\|\tilde{\beta}_n\|_{C^{0,\alpha}(\mathrm{T}_{s^{-1}})}
\le \frac{\nu_{\star}}{2}\le \frac{\epsilon}{5}.
\end{eqnarray*}	

Next, we choose $\nu_{\star}>0$ sufficiently small so that
$$
\nu_{\star} \cdot \left[\left(\frac{|g_1(0)|}{\mu_0}\right)^{\sigma-\gamma} +1 \right]
\le \frac{\epsilon}{5}.
$$

In what follows, we estimate
$$
\|\tilde{\varrho}\|_{L^{\infty}((\Omega_{s^{-1}})_{1})} \cdot (|\tilde{q}|^{(\sigma-\gamma)_+}+1).
$$
If $0 < \sigma \le \gamma$, then $(\sigma-\gamma)_+ =0$ and thus
\begin{eqnarray}
\|\tilde{\varrho}\|_{L^{\infty}((\Omega_{s^{-1}})_{1})}(|\tilde{q}|^{(\sigma-\gamma)_+}+1)
&\le&
\frac{s^{2+\gamma-\sigma}}{K^{1+\gamma-\sigma}} \|\varrho\|_{L^{\infty}(\Omega_{1})}
\left(\left |\frac{g_1(0)}{\tilde{\beta}_n(0)} \right |^{(\sigma-\gamma)_+}+1\right)
\nonumber\\
&\le&
\frac{\nu_{\star}}{2}, \label{7}
\end{eqnarray}
where we used the assumption that $s<1$ together with the definition of $K$.\\
If $\gamma < \sigma \le 1+\gamma$, then $(\sigma-\gamma)_+ = \sigma-\gamma$, and using again that $s<1$ and the definition of $K$ it follows that
\begin{eqnarray}
\|\tilde{\varrho}\|_{L^{\infty}((\Omega_{s^{-1}})_{1})}(|\tilde{q}|^{(\sigma-\gamma)_+}+1)
&\le&
\frac{s^{2+\gamma-\sigma}}{K^{1+\gamma-\sigma}} \|\varrho\|_{L^{\infty}(\Omega_{1})}
\left(\left |\frac{g_1(0)}{\tilde{\beta}_n(0)} \right |^{\sigma-\gamma}+1\right)
\nonumber\\
&\le&
\frac{\nu_{\star}}{4}
\left[\left(\frac{|g_1(0)|}{\mu_0}\right)^{\sigma-\gamma} +1 \right]
\nonumber\\
&\le&
\frac{\nu_{\star}}{2}. \label{8}
\end{eqnarray}
Therefore, from \eqref{7} and \eqref{8}, we obtain
$$
\|\tilde{\varrho}\|_{L^{\infty}((\Omega_{s^{-1}})_{1})} \cdot (|\tilde{q}|^{(\sigma-\gamma)_+}+1)
\le \frac{\epsilon}{5}.
$$ 
This completes \textbf{Step 1}.
\item[\textbf{Step 2.}] \textbf{Iterative affine approximation scheme}
\\
In this step, we show that there exists a sequence of affine functions
$\ell_k(x)=\vec{\mathfrak{b}}_{k}\cdot x$ such that:
\begin{itemize}
\item [i.] $  \|\vec{\mathfrak{b}}_{k+1} - \vec{\mathfrak{b}}_{k}\| \leq  \mathrm{C}_{\star}r^{k\alpha^{\prime}}$,
\item [ii.] $\beta_0 \cdot \vec{\mathfrak{b}}_k =0$,
\item [iii.] $\|u - \ell_{k}\|_{L^{\infty}(\Omega_{r^k})} \leq r^{k(1+\alpha^{\prime})}$,
\end{itemize}
where $r$ is the constant from Lemma \ref{key 1} and
$$
\alpha^{\prime} = \min\left\{\alpha_0^{-}, \alpha,\frac{1}{1+\gamma}\right\}.
$$
We argue by induction. The case $k=0$ follows directly from Lemma \ref{key 1} with
$\ell_0 \equiv 0$. Assume now that the statement holds for some $k \in \mathbb{N}$, and define
the rescaled function
$$
w_{k}(x) = \frac{(u - \ell_{k})(r^{k}x)}{r^{k(1+\alpha^{\prime})}}.
$$
By the induction hypothesis, $w_{k}$ is a normalized solution of
$$
\left\{
\begin{array}{rclcl}
\mathcal{G}_{k}(D^{2}w_{k},Dw_{k}+\vec{q}_{k},x)&=& f_{k}(x) & \text{in} & \left( \frac{1}{r^k} \Omega\right) \cap \mathrm{B}_1, \\
&& &&\\
\beta_{k} \cdot D w_{k} &=& g_{k}(x) & \text{on} & \partial \left( \frac{1}{r^k} \Omega\right) \cap \mathrm{B}_1,
\end{array}
\right.
$$
where $\vec q_k := r^{- k \alpha^{\prime}}(\vec{\mathfrak{b}}_k -\vec q)$ and
\[
\left\{
\begin{aligned}
\mathcal{G}_{k}(\mathrm{M},\vec{\xi},x)&:= |\vec{\xi}|^{\gamma}r^{k(1-\alpha^{\prime})} F\left(r^{k(\alpha^{\prime}-1)} \mathrm{M} \right)+\varrho_{k}(x)|\vec{\xi}|^{\sigma}, \\
\varrho_k(x) &:= r^{k[1-\alpha^{\prime}(1+\gamma-\sigma)]}\varrho(r^kx),\\
f_{k}(x) &:= r^{k(1-\alpha^{\prime}(1+\gamma))} f(r^{k}x), \\
\beta_{k}(x) &:= \beta(r^{k}x), \\
g_{k}(x) &:= r^{-k\alpha^{\prime}}\left(g(r^{k}x) - \beta(r^{k}x) \cdot \vec{\mathfrak{b}}_{k} \right)
=\tilde{g}_k(x)-\tilde{\beta}_k(x),
\end{aligned}
\right.
\]
with $\tilde{g}_k(x)=r^{-k \alpha^{\prime}}g(r^k x)$ and
$\tilde{\beta}_k(x) = r^{-k \alpha^{\prime}} \beta(r^k x)\cdot \vec{\mathfrak{b}}_k$. Since $r < 1$, it follows by construction that
$$
\|f_{k}\|_{L^{\infty}((\Omega_{r^{-k}})_{1})}
\leq r^{k(1-\alpha^{\prime}(1+\gamma)} \|f\|_{L^{\infty}(\Omega_{1})}
\leq \frac{\epsilon}{5},
$$
since $\alpha^{\prime} \leq \frac{1}{1+\gamma}$. 
We now examine the quantity $I_{k}=\|\varrho_{k}\|_{L^{\infty}(\Omega_{1})}(|\vec{q}_{k}|^{(\sigma-\gamma)_{+}}+1)$. To this purpose, we split the analysis into two cases:
\begin{itemize}
\item[{\bf Case 1:}] $0<\sigma\leq \gamma$
\\
Under this condition, $(\sigma-\gamma)_{+}=0$ and consequently,
\begin{eqnarray*}
I_{k}\leq 2r^{k(1-\alpha^{\prime}(1+\gamma-\sigma))}\|\varrho\|_{L^{\infty}(\Omega_{1})}\leq \frac{\epsilon}{5},
\end{eqnarray*}
since $0<r<\frac{1}{2}$ and $1-\alpha^{\prime}(1+\gamma-\sigma)>0$ (recall that $\sigma>0$ and $\alpha^{\prime}\leq 1/(1+\gamma)$).
\item[{\bf Case 2:}] $\gamma<\sigma\leq 1+\gamma$\\
In this case, $(\sigma-\gamma)_{+}=(\sigma-\gamma)\in (0,1]$. So,
\begin{eqnarray}
I_{k}&\leq& r^{k(1-\alpha^{\prime}(1+\gamma-\sigma))}\|\varrho\|_{L^{\infty}(\Omega_{1})}(r^{-k\alpha^{\prime}(\sigma-\gamma)}(|\vec{\mathfrak{b}_{k}}|^{\sigma-\gamma}+|\vec{q}|^{\sigma-\gamma})+1)\nonumber\\
&=&r^{k(1-\alpha^{\prime})}\|\varrho\|_{L^{\infty}(\Omega_{1})}(|\vec{\mathfrak{b}_{k}}|^{\sigma-\gamma}+|\vec{q}|^{\sigma-\gamma}+r^{k\alpha^{\prime}(\sigma-\gamma)}).\label{estdeIk}
\end{eqnarray}
Now, by condition i. in the induction hypothesis, it follows that $|\vec{\mathfrak{b}}_{k}|\leq \frac{\mathrm{C}_{\star}}{1-r^{\alpha^{\prime}}}$, consequently in \eqref{estdeIk} we can conclude that
\begin{eqnarray}
I_{k}&\leq& r^{k(1-\alpha^{\prime})}\|\varrho\|_{L^{\infty}(\Omega_{1})}\left(\left(\frac{\mathrm{C}_{\star}}{1-r^{\alpha^{\prime}}}\right)^{\sigma-\gamma}+|\vec{q}|^{\sigma-\gamma}+r^{k\alpha^{\prime}(\sigma-\gamma)}\right)\nonumber\\
&\leq&r^{k(1-\alpha^{\prime})}\|\varrho\|_{L^{\infty}(\Omega_{1})}\left(2+|\vec{q}|^{\sigma-\gamma}\right)\label{estwithC}\\
&<&\frac{\epsilon}{5}\label{estwithr},
\end{eqnarray}
where, in the estimate \eqref{estwithC}, we used the additional condition $\left(\frac{\mathrm{C}_{\star}}{1-r^{\alpha^{\prime}}}\right)^{\sigma-\gamma}<1$ from Lemma \ref{key 1} and in the estimate \eqref{estwithr} that $0<r<\frac{1}{2}$ and $\alpha^{\prime}<1$.
\end{itemize}
Moreover, since $g(0)=0$, we obtain
\begin{eqnarray*}
\|\tilde{g}_k\|_{L^{\infty}(\mathrm{T}_{r^{-k}})}
&\le&
\|g\|_{C^{0,\alpha}(\mathrm{T}_{1})}
\le \frac{\epsilon}{5}, 	
\end{eqnarray*}
since $\alpha^{\prime}\leq \alpha$ and $0<r<1$.

Now, using $\beta_0 \cdot \vec{\mathfrak{b}}_k =0$ and $|\vec{\mathfrak{b}}_k|\le \frac{\mathrm{C}_{\star}}{1-r^{\alpha^{\prime}}}$, we obtain
$$
\|\tilde{\beta}_k\|_{L^{\infty}(\mathrm{T}_{r^{-k}})}
=
\left \|  r^{-k \alpha^{\prime}}(\beta(r^k x) -\beta_0) \cdot \vec{\mathfrak{b}}_k \right\|_{L^{\infty}(\mathrm{T}_{r^{-k}})}
\le |\vec{\mathfrak{b}}_k| \cdot \|\beta-\beta_0\|_{C^{0,\alpha}(\mathrm{T}_{1})}
\le \frac{\epsilon}{5},
$$
which yields
$$
\|g_k\|_{L^{\infty}(\mathrm{T}_{r^{-k}})}
\le \|\tilde{g}_k\|_{L^{\infty}(\mathrm{T}_{r^{-k}})} + \|\tilde{\beta}_k\|_{L^{\infty}(\mathrm{T}_{r^{-k}})}
    \le \frac{\epsilon}{5} + \frac{\epsilon}{5}
<\epsilon.
$$  
Moreover,
$$
\|\beta_k -\beta_0\|_{C^{0,\alpha}(\mathrm{T}_{r^{-k}})}
\le \|\beta- \beta_0\|_{C^{0,\alpha}(\mathrm{T}_{1})}
\le \frac{\epsilon}{5}.
$$
In addition, let $\psi_k(x)\colon= \psi_{\frac{1}{r^k}\Omega}(x)$. Then $\psi_k(0)=0$,
$D \psi_k(0)=0$, and $\|\psi_k\|_{C^1((\mathrm{T}_{r^{-k}}))} \le \epsilon$. Therefore, we are
under the assumptions of Lemma \ref{key 1} for $\hat{\alpha}=\alpha^{\prime}$, and hence there exists a linear function
$\tilde{\ell}=\vec{\mathfrak{b}} \cdot x$ such that
$$
\|w_k-\tilde{\ell}\|_{L^{\infty}(\frac{1}{r^k} \Omega \cap \mathrm{B}_r)}
\le r^{1+\alpha^{\prime}}, \,\,\,\, \beta_0 \cdot \vec{\mathfrak{b}} =0 \,\,\, \textrm{and} \,\,\,\, |\vec{\mathfrak{b}}| \le \mathrm{C}_{\star}.
$$
Thus, defining
$$
\ell_{k+1}(x)=\ell_k(x) + r^{k(1+\alpha^{\prime})}\tilde{\ell}(r^{-k} x),
$$
we conclude that $\ell_{k+1}$ satisfies conditions i), ii), and iii) with $k$ replaced by $k+1$.
This completes \textbf{Step 2}.
\item[\textbf{Step 3.}] \textbf{Optimal regularity}\\ 
The estimates obtained in \textbf{Step 2} imply that the sequence $(\ell_k)_{k\in\mathbb{N}}$
converges to an affine function $\ell_{\infty}=\displaystyle\lim_{k\to +\infty}\ell_k$, which provides the
first-order expansion of $u$ at $0$. In particular, the decay estimate in item iii) yields that
$u$ is $C^{1,\alpha^{\prime}}$ at $0$. Finally, a standard covering argument allows us to conclude that
$u \in C^{1,\alpha^{\prime}}(\overline{\Omega_{1/2}})$.
\end{enumerate}
\end{proof}

We now prove Corollary \ref{main1} as a direct consequence of Theorem \ref{main}.

\begin{proof}[\bf Proof of Corollary \ref{main1}]
Let $u$ be a viscosity solution of \eqref{2.1}. Define
\[
\bar{g}(x):=g(x)-\zeta(x)u(x), \qquad x\in\partial\Omega_1.
\]
Then $u$ also solves
\begin{equation*}
\left\{
\begin{array}{rclcl}
\mathcal{G}(D^{2}u,Du,x) &=& f(x)& \mbox{in} & \Omega_1,\\
\mathcal{B}(Du,0,x)&=& \bar{g}(x)&\mbox{on}& \partial\Omega_1.
\end{array}
\right.
\end{equation*}
By the structural assumptions, $\bar{g}$ is bounded. Hence, we can invoke the boundary Hölder regularity result from Theorem \ref{Holderreg} to obtain $u\in C^{0,\bar{\alpha}}(\overline{\Omega_{\frac{4}{5}}})$ for some $\bar{\alpha}=\bar{\alpha}(n,\lambda,\Lambda,\mu_{0})\in(0,1)$. In this case, we have that $\beta,\bar{g}\in C^{0,\hat{\alpha}}(\partial\Omega_{\frac{4}{5}})$ for
\[
\hat{\alpha}\defeq\min\{\alpha,\bar{\alpha}\}.
\]
Hence, Theorem \ref{main} yields $u\in C^{1,\hat{\alpha}^{\prime}}(\overline{\Omega_{\frac{2}{5}}})$,
where
\[
\hat{\alpha}^{\prime}
=\min\left\{\alpha_{0}^{-},\hat{\alpha},\frac{1}{1+\gamma}\right\}.
\]
In particular, $u\in C^{0,1}(\overline{\Omega_{\frac{2}{5}}})$. Consequently,
\[
\bar{g}\in C^{0,\alpha}(\partial\Omega_{\frac{2}{5}}).
\]
We can therefore apply Theorem \ref{main} once again for the problem 
\begin{equation*}
\left\{
\begin{array}{rclcl}
\mathcal{G}(D^{2}u,Du,x) &=& f(x)& \mbox{in} & \Omega_\frac{2}{5},\\
\mathcal{B}(Du,0,x)&=& \bar{g}(x)&\mbox{on}& \partial\Omega_\frac{2}{5}.
\end{array}
\right.
\end{equation*}
and conclude the desired regularity of $u$. This completes the proof.
\end{proof}

\begin{remark}
The compactness framework developed in Section \ref{Section3}, together with the boundary improvement-of-flatness scheme in Section \ref{Section4}, provides a flexible approach that can be adapted to other degenerate fully
nonlinear models with lower order terms and oblique boundary conditions.
\end{remark}
%%%%%%%%%%%%%%%%%%%%%%%%%%%%%%%%%%%%%%%%%%%%%%%%%%%%%%%%%%% 

\subsection*{Acknowledgments}

J. da Silva Bessa has been supported by FAPESP-Brazil under Grant No. 2023/18447-3. G.C. Ricarte has been partially supported by
CNPq-Brazil under Grants No. 304239/2021-6.

\subsection*{Declarations}

\subsection*{Conflict of interest}

On behalf of all authors, the corresponding author states that there is no conflict of interest.

\subsection*{Data availability statement}

Data availability does not apply to this article as no new data were created or analyzed in this study.

\bigskip
\noindent\textsc{Junior da Silva Bessa}\\
	Departamento de Química\\
	Faculdade de Educação de Itapipoca - FACEDI\\
	Universidade Estadual do Ceará - UECE\\
	Rua da Universidade, s/n°, CEP 62505-090, Itapipoca, CE, Brazil\\
	and\\
    Departamento de Matem\'{a}tica\\
Instituto de Matem\'{a}tica, Estat\'{i}stica e Computa\c{c}\~{a}o Cient\'{i}fica - IMECC\\
Universidade Estadual de Campinas - Unicamp\\
 Rua S\'{e}rgio Buarque de Holanda, 651, CEP 13083-859, Campinas - SP, Brazil \\
    \noindent \texttt{silva.bessa@uece.br, jbessa@unicamp.br}	\bigskip
	
	\noindent\textsc{Gleydson C. Ricarte}\\
	Departamento de Matemática\\
	Universidade Federal do Cear\'a - UFC \\
	60455-760 Fortaleza, Cear\'a, Brazil \\
	\noindent\texttt{ricarte@mat.ufc.br}

\end{document}